\documentclass[letterpaper,11pt,reqno]{amsart}

\makeatletter
\usepackage{amssymb}
\usepackage{latexsym}
\usepackage{amsbsy}
\usepackage{amsfonts}
\usepackage{hyperref}
\usepackage{graphicx}
\usepackage{enumerate}
\usepackage{enumitem}
\usepackage{mathtools}
\usepackage{color}

\usepackage{tikz}

\def\marginpar#1{\ignorespaces}

\def\Rcup{ \breve{ R } }
\def\Rcap{ \hat{R} }

\DeclareMathOperator\erfc{erfc}

\DeclarePairedDelimiter\floor{\lfloor}{\rfloor}
\def\Sup { S^{\uparrow }}
\def\Sdo { S^{\downarrow }}
\def\Gup { G^{\uparrow }}
\def\Gdo { G^{\downarrow }}
\def\Jup { J^{\uparrow }}
\def\Jdo { J^{\downarrow }}
\def\Ldo { L^{\downarrow }}
\def\Lup { L^{\uparrow }}
\def\Zdo { Z^{\downarrow }}
\def\Zup { Z^{\uparrow }}

\newtheorem{theorem}{Theorem}[section]
\newtheorem{lemma}[theorem]{Lemma}
\newtheorem{proposition}[theorem]{Proposition}
\newtheorem{corollary}[theorem]{Corollary}

\numberwithin{equation}{section}
\makeatother
\begin{document}
\title[Order statistics]{Extreme order statistics of random walks}

\author[Jim Pitman]{{Jim} Pitman}
\address{Department of Statistics, University of California, Berkeley.  
} \email{pitman@stat.berkeley.edu}

\author[Wenpin Tang]{{Wenpin} Tang}
\address{Department of Industrial Engineering and Operations Research, Columbia University. 
} \email{wt2319@columbia.edu}

\date{\today} 
\begin{abstract}
This paper is concerned with the limit theory of the extreme order statistics derived from random walks.
We establish the joint convergence of the order statistics near the minimum of a random walk in terms of the Feller chains.
Detailed descriptions of the limit process are given in the case of simple symmetric walks and Gaussian walks.
\end{abstract}
\maketitle
\textit{Key words :} Bessel processes, Brownian embedding, fluctuation theory, limit theorems, order statistics, path decomposition, random walk.

\smallskip
\maketitle
\textit{AMS 2010 Mathematics Subject Classification:} 60G50, 60F17, 60J05.

\section{Introduction}

\quad Extreme value theory and order statistics have received considerable interest in probability theory and statistics.
They have a variety of applications including nonparametric tests \cite{Lehmann}, kernel density estimation \cite{Silver86}, physics \cite{MOR11, TR77}, environmental sciences \cite{KPN02}, financial risk analysis \cite{EKM97}, and reliability engineering \cite{DN03}.
The most studied case is order statistics derived from a sequence of independent and identically distributed (i.i.d.) random variables \cite{ABN92, DN03, Reiss89}.
We refer to the books \cite{ABN92, DN03, Reiss89} for examples and related results.

\quad In this paper, following earlier work of Pollaczek \cite{Pol52}, Wendel \cite{Wendel60} and others,
we study the order statistics derived from a random walk $S = (S_k, \, 0 \le k \le n)$, with $S_0 = 0$, and $S_k = \sum_{i = 1}^k X_i$ with increments
$X_1, X_2, \ldots, X_n$ that are either {\em exchangeable}, or {\em independent and identically distributed (i.i.d.)}.
For $0 \le k \le n$, let $M_{k,n} = M_k(S_0, \ldots, S_n)$ be the $k^{th}$ order statistics derived from the steps $(S_0, \ldots, S_n)$ of the walk $S$. So
\begin{equation}
\label{eq:Mkn}
\{0 = S_0, S_1, \ldots, S_n \} = \{ M_{k,n} , 0 \le k \le n \} \mbox{ with } M_{0,n} \le M_{1,n} \le \cdots \le M_{n,n}.
\end{equation}
We also use the notations 
\begin{equation}
\label{eq:notdef}
\underline{S}_n: = M_{0,n} := \min_{0 \le k \le n} S_k \qquad \mbox{ and } \qquad \overline{S}_n : = M_{n,n}: = \max_{0 \le k \le n} S_k.
\end{equation}
The fluctuation theory of random walks, developed in the 1950s and 1960s by Sparre Andersen, Spitzer, Baxter, Feller and Wendel, describes the distributions of these order statistics, especially the minimum, the maximum, and related random variables such as the random times when these values are attained, 
either for the first or last time.
Development of this theory was motivated by diverse applications of random walks, 
first in queueing theory \cite{Asmussen03, Pol52} and mathematical statistics \cite{BBBB, GJ14},
and more recently in finance \cite{Dassios95, JYC09}.
The systematic study of large $n$ limit distributions of these functionals of random walks led via Donsker's theorem and its generalizations to the fluctuation theory of Brownian motion and L\'evy processes, as reviewed in the recent monographs \cite{Bertoinbook, Ky06}.

\quad Beside 
this  mature fluctuation theory of random walks and L\'evy processes, 
there has been a recent resurgence of interest in the order statistics of random walks in the physics community.
The papers \cite{SM12, SM14} are typical of this literature, 
starting from particular models of random walk
such as symmetric walks with continuous increment distributions, and making sustained calculations
in these models. 
Many of these calculations recover results that are known in the fluctuation theory of random walks, even if by now somewhat buried in the literature.
But some of these calculations have led to new limit distributions and asymptotic formulas, whose place relative to the standard fluctuation theory does not seem obvious.
As a case in point, Schehr and Majumdar \cite{SM12} focused attention on the spacings or gaps between the random walk order statistics,
\begin{equation}
\label{eq:Dgap}
D_{k,n}:= M_{k,n}- M_{k-1,n}, \quad 1 \le k \le n.
\end{equation}
It was observed by Schehr and Majumdar \cite[(9)]{SM12} that if the distribution of $X_1$ has a symmetric density with finite variance, then as $n \to \infty$ the expected spacing $\mathbb{E}D_{k,n}$ has a limit,
for which they gave an integral expression involving the Fourier transform of the density of $X_1$.
Some known results, reviewed in Section \ref{sc32}, provide explicit descriptions of both the exact 
distribution of $M_{k,n}$, and its asymptotic distributions in various limit regimes as $n \to \infty$. 
But much less is easily found in the literature regarding the large $n$ limit behavior of the differences $D_{k,n}$, which involves the joint distributions of variables in the sequence $(M_{0,n}, \, M_{1,n}, \ldots)$.

\quad Our first result provides a complete description of the joint limit behavior of the order statistics near the minimum of a random walk, relative to the location of the minimum, for any distribution of the increments of the random walk.
The construction relies on the Feller chains $\Sup$ and $\Sdo$, 
introduced by Feller \cite{Feller2}.
The {\em upward Feller chain} $\Sup$ with $\Sup_0 = 0$ is the sequence of partial sums of those increments $X_k$ of the random walk $S$ with $S_k >0$, 
and the {\em downward Feller chain} $\Sdo$ with $\Sdo_0 = 0$ is the sequence of partial sums of those increments $X_k$ of the random walk $S$ with $S_k \le 0$.

\begin{theorem}
\label{thm:main}
Let $(S_k, \, k \ge 0)$ be a real-valued walk with exchangeable increments $X_1, X_2, \ldots$, whose common distribution satisfies $\mathbb{P}(X_1 = 0) < 1$.
For $0 \le k \le n$, let $M_{k,n}$ be the sequence of order statistics defined by \eqref{eq:Mkn} of the $n$-step walk $(S_k, \, 0 \le k \le n)$, and $W_{k,n}: = M_{k, n} - M_{0,n}$.
Let $0 = W_0 \le W_1 \le \cdots$ be the order statistics
\begin{equation}
\label{eq:Wkrep}
W_k:= M_k ( \{ - \Sdo_n, \, n \ge 0 \}  \cup \{\Sup_n, \, n \ge 1 \}),
\end{equation}
derived from the two Feller chains $(\Sdo_n, \, n \ge 0 )$ and $(\Sup_n, \, n \ge 1 )$ generated by the walk $S$.
Then we have the following:
\begin{itemize}
\item[(i)] 
For each finite $K$,  there is the convergence in total variation of finite-dimensional distributions of order statistics 
\begin{equation}
\label{tvconv}
(W_{k,n}, 1 \le k \le K )  \stackrel{TV}{\longrightarrow}( W_k, 1 \le k \le K ) \quad \mbox{as } n \to \infty.
\end{equation}
\item[(ii)]
For each fixed $w > 0$, there is the convergence in total variation of laws of counting processes
\begin{equation}
\label{tvconv2}
\left(\sum_{k=1}^\infty 1( W_{k,n} \le v), \, 0 \le v \le w \right) \stackrel{TV}{\longrightarrow} \left( \sum_{k=1}^\infty 1(W_k \le v),  0 \le v \le w \right) \quad \mbox{as } n \rightarrow \infty.
\end{equation}
\end{itemize}
\end{theorem}

\quad As will be seen in the proof of Theorem \ref{thm:main}, by letting $\alpha_n:= \max\{k \le n: S_k = M_{0,n}\}$, the values $W_k$ derived from $( -\Sdo_n, \, n \ge 0)$ represent contributions to the limiting process from values of $S_j$ with $j$ at or before time $\alpha_n$, while the values $W_k$ derived from $(\Sup_n, \, n \ge 1 )$ are contributions from strictly after time $\alpha_n$.
See also \cite[Theorem 4]{V13} for a related result.

\quad Recall from \eqref{eq:Dgap} that $D_{k,n}:= M_{k,n}- M_{k-1,n} = W_{k,n} - W_{k-1,n}$.
Observe the reversibility of spacings
\begin{equation}
\label{eq:revspacings}
(D_{n,n}, \ldots, D_{1,n})  \stackrel{d}{=}  (D_{1,n}, \ldots, D_{n,n}),
\end{equation}
which follows from the well known duality relation $(S_n - S_k, \, 0 \le k \le n) \stackrel{d} {=}(S_k, \, 0 \le k \le n)$,
implied by the reversibility of increments $(X_n, \ldots, X_1) \stackrel{d}{=} (X_1, \ldots, X_n)$.
Combining Theorem \ref{thm:main} with some known results (see  Lemma \ref{lem:orderid}) regarding the distributions of $M_{k,n}$ yields the following proposition.

\begin{proposition}
\label{prop:main}
With notations as in Theorem \ref{thm:main}, and $X:= X_1$ the generic increment of the walk with exchangeable increments, we have:
\begin{itemize}
\item [(i)]
For each finite $K$, there is the convergence in total variation of finite-dimensional distributions of gaps of order statistics
\begin{equation}
\label{eq:tvconvD}
(D_{k,n}, \, 1 \le k \le K) \stackrel{TV}{\longrightarrow} (D_{k}, \, 1 \le k \le K) \quad \mbox{as } n \rightarrow \infty,
\end{equation}
where $D_k: = W_k - W_{k-1}$. 
Moreover, these convergences hold together with convergence of $p^{th}$ moments for every $0 < p < \infty$ such that $\mathbb{E} |X|^p < \infty$.
\item[(ii)]
If $\mathbb{E}|X| < \infty$, then for each fixed $k = 1,2, \ldots$
\begin{equation}
\label{eq:limexps}
\lim_{n \to \infty } \mathbb{E} D_{k,n}  = \mathbb{E} D_k = \frac{\mathbb{E} S_k^+ }{k} + \mathbb{E}[\mathbb{E}(X|\mathcal{E})^{-}],
\end{equation}
where $\mathcal{E}$ is the exchangeable $\sigma$-field, and $x^{+}:=\max(x,0)$ and $x^{-}:=-\min(x,0)$ for all real $x$.
We have $\lim_{k \to \infty} \mathbb{E}D_k = \mathbb{E}| \mathbb{E}(X|\mathcal{E})|$.
Assume further that $\mathbb{E}X^2 < \infty$ and $\mathbb{E}(X|\mathcal{E}) = 0$ almost surely. 
Then
\begin{equation}
\label{eq:dasym0}
\mathbb{E} D_k \sim \frac{ \mathbb{E}\sqrt{\mathbb{E}(X^2|\mathcal{E})} }{ \sqrt{2 \pi} }  k^{-1/2}  \,\, \mbox{and} \quad \mathbb{E} W_k \sim \mathbb{E}\sqrt{\mathbb{E}(X^2|\mathcal{E})} \sqrt{ \frac{ 2 k}{\pi} } \quad \mbox{as } k \to \infty.
\end{equation}
\end{itemize}
\end{proposition}

\quad The distribution of $(W_k, \, k \ge 0)$ is further simplified in the case of a random walk, due to the following description of the Feller chains in that case, by Bertoin \cite{Bertoin93}.

\begin{theorem}
\label{thm:Bertoin93}
\cite{Bertoin93}
Let $(S_k, \, k \ge 0)$ be a random walk with i.i.d. increments $X_1, X_2, \ldots$, whose common distribution satisfies $\mathbb{P}(X_1 = 0) < 1$.
Then the two Feller chains $\Sup$ and $\Sdo$ are independent, and both are Markov chains, with stationary transition probabilities given by
\begin{align}
& p^{\uparrow}(x, dy): =  \mathbb{P}(x + \Sup_1 \in dy) =  1(y > 0) \, \frac{h^{\uparrow}(y)}{h^{\uparrow}(x)} \, \mathbb{P}(x + X_1 \in dy) , \quad x > 0, \\
& p^{\downarrow}(x, dy): =  \mathbb{P}(x + \Sdo_1 \in dy)=  1(y \le 0) \, \frac{h^{\downarrow}(y)}{h^{\downarrow}(x)} \,\mathbb{P}(x + X_1 \in dy), \quad x \le 0,
\end{align}
where $h^{\uparrow}(x): = \mathbb{E}\left(\sum_{k = 0}^{\tau^{+} - 1}1(S_k > -x) \right)$ with $\tau^{+}: = \inf\{n > 0: S_n > 0\}$, and $h^{\downarrow}(x): = \mathbb{E}\left(\sum_{k = 0}^{\tau^{-0} - 1}1(S_k  \le -x) \right)$ with $\tau^{-0}: = \inf\{n > 0: S_n \le 0\}$.
\end{theorem}

\quad The following corollary is an immediate consequence of Theorem \ref{thm:main} and Tanaka's construction of Feller's chains of a random walk (see Proposition \ref{prop:Tanakaa}).

\begin{corollary}
\label{cor:intro}
For a random walk with i.i.d. increments,  let
$\tau^{+}: = \inf\{n > 0: S_n  > 0\}$ and $\tau^{-0}: = \inf\{n > 0: S_n \le 0\}$, with the conventions that 
$S_{\tau^+} = \infty$ if $\tau^{+} = \infty$ and $-S_{\tau^{-0}} = \infty$ if $\tau^{-0} = \infty$. 
The common distribution of the lowest gap $D_{1,n}:= M_{1,n} - M_{0,n}$ and the topmost gap $D_{n,n}:= M_{n,n} - M_{n-1,n}$
converges in total variation norm as $n \to \infty$ to the distribution of $D_1 = W_1$ defined by
\begin{equation}
\label{eq:Wdecomp}
\mathbb{P}( W_1 > w) = \mathbb{P}(S_{\tau^+} >w)  \, \mathbb{P}(- S_{\tau^{-0}} >w) \quad w > 0.
\end{equation}
\end{corollary}

\quad In the case of a random walk, Proposition \ref{prop:main} specializes to the following corollary.

\begin{corollary}
\label{coro:main}
Let $(S_k, \, k \ge 0)$ be a random walk with i.i.d. increments $X_1, X_2, \ldots$, whose common distribution satisfies $\mathbb{P}(X_1 = 0) < 1$.
Let $X: = X_1$ be the generic increment of the random walk.
We have: 
\begin{itemize}
\item [(i)]
Proposition \ref{prop:main} (i) holds.
\item[(ii)]
If $\mathbb{E}|X| < \infty$ and $\mathbb{E} X = \mu$, then for each fixed $k = 1,2, \ldots$
\begin{equation}
\label{eq:limexps2}
\lim_{n \to \infty } \mathbb{E} D_{k,n}  = \mathbb{E} D_k = \frac{\mathbb{E} S_k^+ }{k} + \mu^{-}.
\end{equation}
Assume further that $\mu = 0$ and $\mathbb{E}X^2 = \sigma^2 < \infty$. 
Then
\begin{equation}
\label{eq:dasym2}
\mathbb{E} D_k \sim \frac{ \sigma }{ \sqrt{2 \pi} }   \, k^{-1/2}  \quad \mbox{and} \quad \mathbb{E} W_k \sim \sigma \sqrt{ \frac{ 2}{\pi} } k ^{1/2} \quad \mbox{as } k \to \infty.
\end{equation}
\end{itemize}
\end{corollary}

\quad As mentioned earlier, Schehr and Majumdar \cite{SM12} showed the existence of the limit \eqref{eq:limexps2} for symmetric distributions of $X$ with a density, with a more complicated description of the limit by a Fourier integral. 
Their result reduces to the simpler formula in \eqref{eq:limexps2} by application of the formula 
\cite{Brown70, Hsu51, vB65} 
\begin{equation*}
\mathbb{E} |X| = \frac{2}{\pi} \int_{0}^\infty ( 1 - \mathbb{E} e^{i t X } ) \frac{ dt } {t^2} 
\end{equation*}
for any symmetric distribution of $X$.

\quad Another interesting question is how fast $(W_{k,n}, \, 1 \le k \le K)$ and $(D_{k,n}, \, 1 \le k \le K)$ converges as $n \rightarrow \infty$ to the limiting processes $(W_{k}, \, 1 \le k \le K)$ and $(D_{k}, \, 1 \le k \le n)$ respectively.
The following proposition quantifies a convergence rate in \eqref{tvconv} and \eqref{eq:tvconvD} for a random walk with i.i.d. increments which is continuous, and satisfies some moment condition.
We leave the question of whether the established rate is optimal as open problem.

\begin{proposition}
\label{prop:mainrate}
Let $(S_k, \, k \ge 0)$ be a random walk with i.i.d. increments $X_1, X_2, \ldots$, whose distribution is continuous and $\mathbb{E} |X_1|^2 < \infty$.
With notations as in Theorem \ref{thm:main} and Proposition \ref{prop:main}, let $\mbox{TV}(K,n)$ be the total variation distance between the distributions of the two sequences in \eqref{tvconv} or \eqref{eq:tvconvD}.
Then for each finite K, there is a constant $C_K > 0$ such that
\begin{equation}
\label{eq:boundn}
\mbox{TV}(K, n) \le C_K \,   n^{-\frac{1}{4}}.
\end{equation}
\end{proposition}

\quad It is well known that in the setting of Corollary \ref{cor:intro}, the distributions $F_+$ and $F_{-0}$ of
the strict ascending and weak descending ladder heights $S_{\tau^{+}}$ and $S_{\tau^{-0}}$ with supports in 
$(0,\infty)$ and $(-\infty,0]$ respectively,
are uniquely determined by the Wiener-Hopf equation $F_+ + F_{-0} - F_+ * F_{-0} = F$,
where $F$ is the distribution of increments.
The joint limit law of  $(W_1,W_2)$ can also be described just as explicitly for a general increment distribution, and in principle so can that of
$(W_1,W_2,W_3)$ and so on, but these general descriptions become more complex as the number of variables increases.
Three special cases for which much more detailed descriptions are possible are:
\begin{itemize}[itemsep = 3 pt]
\item[$(i)$]
simple random walk, with increments $X_i = \pm 1$ and $\mathbb{P}(X_i = 1) = \mathbb{P}(X_i = -1) = 1/2$; 
\item[$(ii)$]
symmetric Gaussian random walk, whose increments have the common distribution $\mathbb{P}(X_i \in dx) = (2 \pi \sigma^2)^{-\frac{1}{2}} \exp\left(-\frac{x^2}{2 \sigma^2} \right) dx$ for some $\sigma^2 > 0$;
\item[$(iii)$]
symmetric Laplacian random walk, whose increments have the common distribution $\mathbb{P}(X_i \in dx) = (2b)^{-1} e^{-b|x|} dx$ for some $b > 0$.
\end{itemize}

\quad In this paper we will illustrate results for walks with simple $\pm 1$ and symmetric Gaussian increments.
The case of symmetric Laplacian increments involves further symmetries and will be treated
in a separate paper \cite{Pitman20}.
The key to the analysis of these random walks is to embed them into a Brownian motion.
By the Skorokhod embedding \cite{Sk65}, every random walk with mean zero and finite variance $\sigma^2$ per step may be embedded in a Brownian motion $(B(t), \, t \ge 0 )$ as $S_k = B(T_k)$
where $0 = T_0 \le T_1 \le \cdots$ is an increasing sequence of stopping times of $B$, with the $(T_k- T_{k-1}, \, k \ge 1)$ i.i.d. as $T_1$, and $\mathbb{E}T_k = k \sigma^2$. 
In the above three examples, 
\begin{itemize}[itemsep = 3 pt]
\item[$(i)$]
Simple symmetric walk, with $T_k = \inf\{t > T_{k-1}: |B_t - B_{T_{k-1}}| = 1\}$; 
\item[$(ii)$]
symmetric Gaussian random walk, with $T_k = \sigma k$;
\item[$(iii)$]
symmetric Laplacian random walk, with $T_k = 2b^{-2} \gamma_k$, where $\gamma_k = \sum_{i = 1}^{k} \varepsilon_i$ for a sequence of independent standard exponential variables.
\end{itemize}

The upward Feller chain $(\Sup_k, \, k \ge 0)$ derived from a simple symmetric walk was studied by Pitman \cite{Pitman75} 
who gave a number of constructions of this Markov chain, including the embedding
$\Sup_k = R_3(T_k)$, $k \ge 0$ 
where $R_3$ is a BES(3) process, and $T_k$ is the succession of random times at which $R_3$ hits integer values, defined inductively by $T_0 := 0$ and $T_{k+1}:= \inf \{t > T_{k}:  | R_3(t) -  R_3(T_k)| = 1 \}$.
It follows from this embedding, and the easily established almost sure convergence $T_n/n \to 1$ as $n \to \infty$, that there is the scaling limit
\begin{equation}
( n^{-1/2} \Sup_{un}, \, u \ge 0 ) \stackrel{d}{\longrightarrow} (R_3(u), \, u \ge 0 )
\end{equation}
in the sense of functional limit theorems in the path space $C[0,\infty)$, that is the same sense in which $( n^{-1/2} S_{ u n}, \, u \ge 0 )$ converges in distribution to Brownian motion. 
See \cite{HKK03} for embedding the Feller chain $(\Sup_k, \, k \ge 0)$ derived from a random walk with a more general increment distribution in a BES(3) process.

\smallskip
{\bf Organization of the paper:} 
In Section \ref{sc2}, we recall basic constructions of the Feller chains.
In Section \ref{sc3}, we study the extreme order statistics derived from general random walks.
There Theorem \ref{thm:main}, Proposition \ref{prop:main}, Corollary \ref{cor:intro} and Proposition \ref{prop:mainrate} are proved.
Section \ref{sc4} provides detailed descriptions of the limit process of order statistics derived from simple symmetric random walks. 
Section \ref{sc5} describes the limit process of order statistics derived from Gaussian random walks. 

\section{Feller chains -- Old \& New}
\label{sc2}

\quad In this section we present some basic constructions in the theory of conditioned random walks.
This theory was systematically developed in the 1980's and 1990's by Tanaka \cite{Tanaka89}, Bertoin and Doney \cite{Bertoin93, BD94} following earlier contributions by Feller \cite{Feller2}, Williams \cite{Williams70, Williams} and Pitman \cite{Pitman75}.
Conditioned random walks also appear in the study of localization of random polymers, see e.g. Comets \cite[Chapter 7]{Comets}.

\quad Let $S_0 = 0$ and $S_n:= X_1 + \cdots + X_n$ be a real-valued walk with exchangeable increments $X_1, X_2, \ldots$.
Let $(\Sup_k, \, k \ge 0)$ and $(\Sdo_k, \, k \ge 0)$ be the upward and downward Feller chains derived from the walk $(S_k, \, k \ge 0 )$ defined in the introduction.
These random sequences with a finite time horizon $n$ were introduced by Feller \cite[XII.8, Lemma 3]{Feller2} to provide a combinatorial proof of Sparre Andersen's equivalence principle \cite{SparreA}.
Formally, the upward Feller chain $S^{\uparrow}$ is the sequence of partial sums of those increments $X_k$ of the walk $S$ with $S_k > 0$, and the downward Feller chain $S^{\uparrow}$ is the sequence of partial sums of those increments $X_k$ of the walk $S$ with $S_k \le 0$
Let $N_n^{+}: = \#\{k \le n: S_k > 0\}$ and $N_n^{-}: = n - N_n^{+}$.
The above construction gives partial sum processes $(\Sup_k, \, 0 \le k \le N^{+}_n)$ and $(\Sdo_k, \, 0 \le k \le N^{-}_n)$ of random lengths $N^{+}_n$ and $N^{-}_n$ respectively.
Then there are identities
\begin{equation}
\label{eq:krecovery}
N_k^{+} = N_{k-1}^{+} + 1(S_k > 0), \quad N_k^{-} = N_{k-1}^{-} + 1(S_k \le 0), \quad S_k = \Sup_{N_k^{+}} + \Sdo_{N_k^{-}}.
\end{equation}
It follows that for each fixed $n$, the original sequence $(S_k, \, 0 \le k \le n )$ is encoded as a measurable function of the two Feller chains $(\Sup_k, \, 1 \le k \le N^{+}_n)$ and $(\Sdo_k, \, 1 \le k \le N^{-}_n)$.
The final values of the two chains give the value of $S_n = \Sup_{N_n^{+}} + \Sdo_{N_n^{-}}$.
This also yields the value of $1(S_n >0)$, hence the values of $N^{+}_{n-1}$ and $N^{-}_{n-1}$ by \eqref{eq:krecovery}
for $k = n$, then the value of $S_{n-1}$ by \eqref{eq:krecovery} for $k = n-1$, and so on recursively down to the value
of $S_1$.
\begin{lemma} \cite{Feller2}
\label{lem:Fellern}
Assume that $(S_k, \, 0 \le k \le n)$ is a walk with exchangeable increments, and 
let $\alpha_n : = \max \left\{k \le n: S_k =\underline{S}_n \right\}$.
Whatever the common distribution of exchangeable increments of $S$, there is the identity in distribution
\begin{multline}
\label{eq:FBid}
\left( (S_{\alpha_n + k} -S_{\alpha_n}, \, 0 \le k \le n - \alpha_n), \,  (S_{\alpha_n - k} - S_{\alpha_n}, \, 0 \le k \le \alpha_n) \right) \\
\stackrel{d}{=} \left((\Sup_k, \, 0 \le k \le N^{+}_n), \,  (-\Sdo_k, \, 0 \le k \le N^{-}_n) \right).
\end{multline}
Assume further that the increments $X_1, X_2, \ldots$ are independent.
\begin{itemize}[itemsep = 3 pt]
\item[(i)] 
The common distribution of $\alpha_n$ and $N^{-}_n$ is given by
\begin{equation}
\mathbb{P}(\alpha_n = \ell) = \mathbb{P}(N^{-}_n = \ell) = \mathbb{P}\left(\cap_{k = 1}^\ell \{S_k \le 0\}\right) 
\mathbb{P}\left(\cap_{k = 1}^{n -\ell} \{S_k > 0\}\right).
\end{equation}
\item[(ii)]
Given $\alpha_n = \ell$, the common distribution of each pair of path fragments displayed in \eqref{eq:FBid} is that of two independent path fragments distributed as $(S_k, \, 0 \le k \le n -\ell)$ given $\cap_{k = 1}^{n -\ell} \{S_k > 0\}$
and $(-S_k, \, 0 \le k \le \ell)$ given $\cap_{k = 1}^{\ell} \{S_k \le 0\}$ respectively.
\end{itemize}
\end{lemma}

\quad Here the main results involve Feller's chains with infinite time horizon.
Observe that the first $n$ steps of the walk $(S_k, \, 0 \le k \le n)$ define only the first $N^{+}_n$ steps of $\Sup$ and the first $N^{-}_n$ steps of $\Sdo$.
As $n \uparrow \infty$, it is obvious that $N^{+}_n \uparrow N^{+}_{\infty}$ and $N^{-}_n \uparrow N^{-}_{\infty}$ for some limiting random variables $N^{+}_{\infty}$ and $N^{-}_{\infty}$ with values in $\{0,1, \ldots, \infty\}$.
When $N^{+}_{\infty} = \infty$ there are infinitely many $X_k$ with $S_k > 0$, and the infinite horizon upward chain $\Sup$ has this infinite list of $X$-values as increments.
Whereas if $N^{+}_{\infty} < \infty$, there is the convention that $\Sup_k = \infty$ for $k > N^{+}_{\infty}$.
So the random variable $N^{+}_{\infty}$ is encoded in $\Sup$ as $N^{+}_{\infty} = \inf\{n: \Sup_{n+1} = \infty\}$ with the convention that $\inf \emptyset = \infty$.
Similar remarks apply to the definition of $\Sdo$.
The following lemma summarizes some basic facts about Feller's chains.

\begin{lemma} \cite{Bertoin93, Feller2}
\label{lem:Fellerinf}
Assume that $(S_k, \, k \ge 0)$ is a real-valued walk with exchangeable increments $X_1, X_2, \ldots$, 
whose common distribution satisfies $\mathbb{P}(X_1 = 0) < 1$.
\begin{itemize}[itemsep = 3 pt]
\item[(i)] There is the convergence in total variation as $n \rightarrow \infty$
\begin{equation}
\label{eq:Fellertvn}
\left( (S_{\alpha_n + k} -S_{\alpha_n}, \, 0 \le k \le n - \alpha_n), \,  (S_{\alpha_n - k} - S_{\alpha_n}, \, 0 \le k \le \alpha_n) \right) \stackrel{TV}{\longrightarrow} \left(\Sup, -\Sdo \right),
\end{equation}
where $\alpha_n : = \max \left\{k \le n: S_k = \underline{S}_n \right\}$.
\item[(ii)] Each Feller chain is upwardly transient:
\begin{equation*}
\lim_{n\to \infty} \Sup_n = \lim_{n\to \infty} - \Sdo_n = \infty \quad \mbox{almost surely}.
\end{equation*}
\end{itemize}
Assume further that the increments $X_1, X_2, \ldots$ are independent.
We have the following:
\begin{itemize}[itemsep = 3 pt]
\item[(iii)] 
The two Feller chains $\Sup$ and $\Sdo$ are independent. 
\item[(iv)]
Each of the Feller chains $\Sup$ and $-\Sdo$ is a Markov chain with stationary transition probabilities, with state space $[0,\infty]$ and $\infty$ as an absorbing state.
\item[(v)]
The chain $\Sup$ is a Doob $h$-transform of the original random walk $S$ stopped on first reaching $(-\infty,0]$, for the super-harmonic function $h^{\uparrow}(x): = \mathbb{E}\left(\sum_{k = 0}^{\tau^{+} - 1}1(S_k > -x) \right)$ 
where $\tau^{+}: = \inf\{n > 0: S_n > 0\}$;
while the chain $\Sdo$ is a Doob $h$-transform of the original random walk $S$ stopped on first reaching $(0, \infty)$,
for the super-harmonic function $h^{\downarrow}(x): = \mathbb{E}\left(\sum_{k = 0}^{\tau^{-0} - 1}1(S_k  \le -x) \right)$
where $\tau^{-0}: = \inf\{n > 0: S_n \le 0\}$.
\end{itemize}
\end{lemma}
Note that neither \cite{Bertoin93} nor \cite{Feller2} mentions the convergence in total variation \eqref{eq:Fellertvn}.
It is only stated in \cite{Bertoin93} that the left side of \eqref{eq:Fellertvn} converges in distribution to its right side.
However, the argument in \cite{Bertoin93}, as is clear in the proof of Theorem \ref{thm:main}, implies the convergence in total variation.

\quad An interesting issue which does not seem to have been addressed in the literature is whether the path of the original walk $(S_k, \, k \ge 0)$ can be recovered almost surely from the pair of infinite horizon Feller chains $(\Sup, \Sdo)$.
While this turns out to be the case, it is not a trivial consequence of the fact that the walk $\{S_k, \, 0 \le k \le n\}$
can be recovered from the finite horizon chain segments $(\Sup_k, \, 1 \le k \le N^{+}_n)$ and $(\Sdo_k, \, 1 \le k \le N^{-}_n)$.
As discussed before, the finite segment recovery involves a reverse induction using the values of $N^{+}_n$ and $N^{+}_n$
to progressively determine the values of $N^{+}_k$ and $N^{-}_k$, and hence $S_k = \Sup_{N^{+}_k} + \Sdo_{N^{-}_k}$ for smaller values of $k$.
But given the entire paths of the two infinite horizon chains, there is no obvious way to determine for any particular $n$ how the $n$ steps are split into $N^{+}_n$ and $N^{-}_n$, so no obvious beginning for an inductive recovery.
This difficulty is avoided by an alternative expression for the recovery of $(S_k, \, k \ge 0)$ from $(\Sup_k, \, k \ge 1)$ and $(\Sdo_k, \, k \ge 1)$ via the ascending and descending chain segments.
To illustrate, we describe how $(S_k, \, 0 \le k \le n)$ is recovered from $(\Sup_k, \, 1 \le k \le N^{+}_n)$ and $(\Sdo_k, \, 1 \le k \le N^{-}_n)$ by such an encoding.

\quad In terms of the segment $(\Sup_k, \, 0 \le k \le N^{+}_n)$, define the sequence of {\em future minimum times} inductively by $\Gup_0 = 0$ and for $j \ge 1$,
\begin{equation}
\label{eq:fmt}
\Gup_j := \max \left\{ k > \Gup_{j-1} : \Sup_k  = \min_{\Gup_{j-1} < i \le N^{+}_n } \Sup_i \right\},
\end{equation}
with the convention that $\max \emptyset = \infty$.
Let $\Jup_n: = \max\{j: \Gup_j < \infty\}$ be the total number of future minimum times derived from the segment $(\Sup_k, \, 0 \le k \le N^{+}_n)$.
So $\Gup_j < \infty$ (which is well-defined) if and only if $\Jup_n \ge j$.
By construction, $\Gup_1$ is the last time that the minimum of $\Sup$ over $\{1,\ldots, N^+_n\}$ is attained.
If $\Gup_1 = N^+_n$ then $\Jup_n= 1$, else $\Gup_2$ is the last time after $\Gup_1$ that the minimum of $\Sup$ over $\{\Gup_1 + 1, \ldots, N^+_n\}$ is attained. 
If $\Gup_2 =N^+_n $ then $\Jup_n = 2$, else $\cdots$, and so on.
Now the idea is to decompose $(\Sup_k, \, 0 \le k \le N^{+}_n)$ into the sequence of {\em ascending chain segments} 
\begin{equation}
\label{eq:fms}
\left(\Sup _{\Gup_{j-1} + k }, \, 0 \le k \le \Gup_{j} - \Gup_{j-1} \right), \quad 1 \le j \le \Jup_n.
\end{equation}
Note that in this construction,
\begin{itemize}[itemsep = 3 pt]
\item[$(i)$] 
$0 = \Gup_0 < \Gup_1 < \cdots < \Gup_{\Jup_n} = N^{+}_n$;
\item[$(ii)$]
$0 = \Sup_{\Gup_0} < \Sup_{\Gup_1} < \cdots < \Sup_{\Gup_{\Jup_n}} = \Sup_{N^{+}_n}$;
\item[$(iii)$]
the increments of each segment of the original walk $(S_k, \, 0 \le k \le n)$ over positive excursions, i.e. segments $(S_{U+i}, \, 0 \le i \le V)$ with $S_{U} \le 0$ and $S_{U +k} > 0$ for $k \le V$, and either $S_{U+V+1} \le 0$ 
or $U+V = n$ and $S_n > 0$, are encoded in $\Sup$ as the concatenation of increments of some consecutive sequences of future minimum segments.
\end{itemize}
Now a similar construction with $\Sdo$ yields a sequence of {\em future maximum times} $\Gdo_j$, $0 \le j \le \Jdo_n$,
and a corresponding sequence of {\em descending chain segments} derived from $(\Sdo_k, \, 0 \le k \le N^{-}_n)$.
The increments of the original walk over negative excursions are encoded in $\Sdo$, with analogs of (i) -- (iii) above, except that if $\mathbb{P}(S_k = 0)>0$ for some $k$ there is the complication that instead of all strict inequalities in (ii), the first inequality might be an equality, so the analog of (ii) is slightly different:
\begin{itemize} 
\item [$(ii')$] $0 = \Sdo_{\Gdo_0} \ge  \Sdo_{\Gdo_1} > \cdots > \Sdo_{\Gdo_{\Jup_n}} = \Sdo_{N^{-}_n}$.
\end{itemize}
The issue of recovering the first $n$ segments of the original walk $(S_k, \, 0 \le k \le n)$ from the two finite chains 
$(\Sup_k, \, 1 \le j \le N^{+}_n)$ and $(\Sdo_k, \, 1 \le j \le N^{-}_n)$ is now framed as one of determining the order in which the $\Jup_n$ ascending chain segments and the $\Jdo_n$ descending chain segments are riffled together as they appeared in the original path.
According to the finite $n$ recovery by reverse induction, there is one and only one possible order to do this riffling.
Moreover, it can be seen directly from the ascending and descending chain segments what this order is:
it is the only order of arrangement of these segments in which the final values of the segments are weakly increasing in absolute value, with the additional requirement that if two segments have the same absolute final value, the ascending segment is placed first.

\quad The extension of the definition of ascending and descending chain segments to the infinite horizon is obvious.
The result is stated as follows.
\begin{proposition}
\label{prop:Tanakaa}
Let $(S_k, \, k \ge 0)$ be a real-valued walk, and $(\Sup_k, \, 0 \le k \le N^{+}_{\infty})$ and $(\Sdo_k, \, 0 \le k \le N^{-}_{\infty})$ be the associated upward and downward Feller chain respectively. 
Define the future minimum times $\Gup$ of $\Sup$ as in \eqref{eq:fmt}, and the sequence of ascending chain segments as in \eqref{eq:fms} (by replacing $\Jup_n$ with $\Jup_{\infty}$); 
similarly define the future maximum times $\Gdo$ of $\Sdo$, and the sequence of descending chain segments.
Then the increments of $S$ are recovered from its sequence of ascending and descending chain segments 
according to the rule:
\begin{itemize}[itemsep = 3 pt]
\item[(i)]
Place the increments from these segments in the order of least absolute final values of the segments.
\item[(ii)]
If two segments have the same absolute final value, place the ascending segment first.
\end{itemize}
Assume further that $(S_k, \, k \ge 0)$ is an oscillating random walk with i.i.d. increments whose distribution is continous.
Then $N^{+}_{\infty} = N^{-}_{\infty} = \infty$ almost surely, and there are the identities in distribution
\begin{equation}
\left(\left(\Gup_j, \Sup_{\Gup_j} \right), \, j = 1,2, \ldots \right) \stackrel{d}{=} \left( \left(\tau^{+}_j, S_{\tau^{+}_j} \right), \,  j = 1,2, \ldots \right),
\end{equation}
\begin{equation}
\left( \left(\Gdo_j, \Sdo_{\Gdo_j} \right), \,  j = 1,2, \ldots \right) \stackrel{d}{=} \left( \left(\tau^{-0}_j, S_{\tau^{-0}_j} \right), \, j = 1,2, \ldots \right),
\end{equation}
where $( (\tau^{+}_j, S_{\tau^{+}_j} ), \,  j = 1,2, \ldots )$ is the sequence of strictly ascending ladder indices and heights of $S$, and $( (\tau^{-0}_j, S_{\tau^{-0}_j} ), \, j = 1,2, \ldots )$ is the sequence of weakly descending ladder indices and heights of $S$.
In particular, $(\tau^{+}_j, S_{\tau^{+}_j} )$ is the sum of $j$ independent copies of the first strictly ascending ladder index and height $(\tau^{+}, S_{\tau^{+}})$ with $\tau^{+}: = \inf\{n>0: S_n > 0\}$, 
and $(\tau^{-0}_j, S_{\tau^{-0}_j} )$ is the sum of $j$ independent copies of the first weakly descending ladder index and height $(\tau^{-0}, S_{\tau^{-0}})$ with $\tau^{-}: = \inf\{n>0: S_n \le 0\}$.
\end{proposition}
\begin{proof}
The first part of this proposition follows the finite $n$ recovery algorithm described in the preceding paragraph.
The second part is a consequence of Tanaka's construction \cite{Tanaka89} of Feller's chains $(\Sup, \Sdo)$.
\end{proof}

\section{Order statistics of general random walks}
\label{sc3}
\quad In this section we prove the main results, Theorem \ref{thm:main}, Proposition \ref{prop:main}, Corollary \ref{cor:intro} and Proposition \ref{prop:mainrate} regarding the order statistics of general random walks.

\subsection{Proof of Theorem \ref{thm:main}}
\label{sc31}
Let $0 =  \widetilde{S}_{0,n}, \widetilde{S}_{1,n}, \ldots, \widetilde{S}_{n,n}$ be the random sequence with increments obtained by concatenation
of the first $N^{-}_n$ increments $(X^{\downarrow}_{1}, \ldots, X^{\downarrow}_{N^{-}_n})$ of $\Sdo$ in reverse order, followed by the first $N^{+}_n$ increments $(X^{\uparrow}_1, \ldots, X^{\uparrow}_{N^{+}_n})$ of $\Sup$ in their original order:
\begin{equation*}
\left(\widetilde{X}_{1,n}, \ldots, \widetilde{X}_{n,n} \right):= \left(X^{\downarrow}_{N^{-}_n}, \ldots, X^{\downarrow}_1, X^{\uparrow}_1, \ldots, X^{\uparrow}_{N^{+}_n}\right).
\end{equation*}
By Lemma \ref{lem:Fellern}, $\left( \widetilde{S}_{k,n}, \, 0 \le k \le n \right) \stackrel{d}{=} (S_k, \, 0 \le k \le n)$.
Regard the sequence $(W_{k,n}, \, 0 \le k \le n)$ as a measurable function, say $\Phi_n$, of $(S_k, \, 0 \le k \le n)$.
Then for each fixed $n$ there is the identity of joint distributions
\begin{equation}
\label{eq:Wid}
(W_{k,n}, \, 0 \le k \le n) \stackrel{d}{=} \left(\widetilde{W}_{k,n}, \, 0 \le k \le n \right),
\end{equation}
where the r.h.s. is $\Phi_n(\widetilde{S}_{k,n}, \, 0 \le k \le n)$.

\quad Observe that as $n$ varies, the two sequences displayed in \eqref{eq:Wid} develop by very different stochastic mechanisms.
For instance, on the l.h.s., as $n$ increments to $n+1$, if $X_{n+1}$ is sufficiently large negative, leading to a sum $S_{n+1}$ with $S_{n+1} = M_{0,n} - a$ with $a>0$, 
then $M_{0,n+1} = M_{0,n} - a$, so $W_{1,n} = a$ and $W_{k+1,n+1} = a + W_{k,n}$ for every $1 \le k \le n$.
In particular, if the process of downward ladder indices of $S$ is recurrent, this will happen infinitely often, and for any fixed $K$ as $n$ varies the values of $(W_{k,n}, \,0 \le k \le K)$ will keep changing as $n \to \infty$. 
In contrast, the right side of \eqref{eq:Wid} evolves in a completely different, and much simpler way, which is the key point of this argument.
As $n$ increments to $n+1$, all that happens on the r.h.s. of \eqref{eq:Wid} is that one of the two partial Feller chains 
$(\Sup_k, \, 0 \le k \le N^{+}_n)$  and $(\Sdo_k, \, 0 \le k \le N^{-}_n)$ has its length incremented by $1$, 
with the consequence that $(\widetilde{X}_{1,n+1}, \ldots, \widetilde{X}_{n+1,n+1})$ is obtained from $(\widetilde{X}_{1,n}, \ldots, \widetilde{X}_{n,n})$ by either prepending or appending $X_{n+1}$ according to whether $S_{n+1} \le 0$ or $S_{n+1} > 0$.
If $S_{n+1} \le 0$ then $(\widetilde{W}_{k,n}, \, 1 \le k \le n +1) $ is derived by insertion of the value 
$-\Sdo_{N^{-}_{n+1}}$ into the previous list $(\widetilde{W}_{k,n}, 1\,  \le k \le n)$,
while if $S_{n+1} > 0$ then $(\widetilde{W}_{k,n}, \, 1 \le k \le n +1) $ is derived by insertion of the value $\Sup_{N^{+}_{n+1}}$ into the previous list.
In either case, the previous list is updated by insertion of a single element, which finds its place somewhere in the previous list.
All  lower values in the list remain unchanged, while the indexing of all higher values in the list is pushed up by $1$. 
In terms of the associated point processes on $[0,\infty)$, all that happens is that a single point is inserted somewhere in the configuration of $n$ points to make a new configuration of $n+1$ points.
As a consequence, for each fixed $k$, with probability one, 
\begin{equation}
\label{eq:noninc}
\mbox{$\widetilde{W}_{k,n}$ for $n \ge k$ is weakly decreasing in $n$ and eventually equal to its limit $W_k$. }
\end{equation}

\quad To be more precise, let 
\begin{equation}
\label{eq:M>n}
M_{>n}: = \min\left(\{\Sup_j, \, j \ge N^{+}_n\} \cup \{-\Sdo_j, \, j \ge N^{-}_n\} \right).
\end{equation}
Let $\mbox{TV}(K,n)$ be the total variation distance between the distributions of the two sequences displayed in \eqref{tvconv}. 
Then in view of \eqref{eq:Wid} there is the coupling bound \cite[(2.6), p.12]{Lind92}:
\begin{equation}
\label{eq:coupling}
\mbox{TV}(K,n) \le  \mathbb{P}( \widetilde{W}_{n,k} \ne W_k \mbox{ for some } 1 \le k \le K ) \le  \mathbb{P}( M_{>n} \le W_K ),
\end{equation}
since if $M_{>n} >  W_K$ then every insertion after time $n$ does not change the value of any $\widetilde{W}_{k,n}$ with $1 \le k \le K$, and hence $\widetilde{W}_{k,m} = W_k$ for all $m \ge n$.
Now by Lemma \ref{lem:Fellerinf}, the sequences $\Sup$ and $-\Sdo$ are upward transient, which implies that
$\mathbb{P}(M_{>n} \uparrow \infty) = 1$.
So the probability bound in \eqref{eq:coupling} decreases to $0$ as $n \to \infty$.
This gives part $(i)$. 
The proof of $(ii)$ is in the same vein.

\subsection{Proof of Proposition \ref{prop:main}}
\label{sc32}
The key to the analysis of Proposition \ref{prop:main} relies on the following representation of the order statistics $M_{k,n}$, which is due to Pollaczek \cite{Pol52} and Wendel \cite[Theorem 2.2(a)]{Wendel60}.

\begin{lemma} \cite{Pol52, Wendel60}
\label{lem:orderid}
Let $(S_k, \, 0 \le k \le n)$ be a real-valued walk with exchangeable increments, 
and let $M_{k,n}$ be the order statistics defined by \eqref{eq:Mkn}.
For each $0 \le k \le n$, the distribution of $M_{k,n}$ is the convolution of the distributions of maximal and minimal values over segments of the walk of length $k$ and $n-k$ respectively:
\begin{equation}
\label{eq:pwd}
M_{k,n} \stackrel{d}{=} \overline{S}_k + \underline{S}'_{n-k}
\end{equation}
where $S'$ is copy of $S$ and independent of $S$.
\end{lemma}

\quad Pollaczek proved Lemma \ref{lem:orderid} for a random walk with i.i.d. increments, and attributed this result to Bohnenblust, Spitzer and Welsh. 
Later Wendel \cite{Wendel60} and Port \cite[Section V]{Port64} provided more combinatorial proofs.
As pointed out by Dassios \cite{Dassios96}, Wendel's argument carries over to any walk with exchangeable increments. Continuous analogs of \eqref{eq:pwd} on the quantile of stochastic processes were studied in \cite{Dassios95, Dassios96, ERY}.

\quad We also recall a formula for the expectation of $\max_{0 \le k \le n} S_k$ and $\min_{0 \le k \le n} S_k$, which is a simple consequence of Spitzer's identity \cite[Theorem 3.1]{Spitzer56} for $S$ a random walk with i.i.d. increments.
It is lesser known that this formula also holds for a walk with exchangeable increments \cite{BS73, SD76}.

\begin{lemma} \cite{SD76, Spitzer56}
\label{lem:expmm}
Let $(S_k, \, 0 \le k \le n)$ be a real-valued walk with exchangeable increments. 
Then
\begin{equation}
\label{eq:spitzer}
\mathbb{E} \overline{S}_n = \sum_{k = 1}^n \frac{\mathbb{E}S_k^{+}}{k} \quad \mbox{and} \quad 
\mathbb{E} \underline{S}_n = - \sum_{k = 1}^n \frac{\mathbb{E}S_k^{-}}{k},
\end{equation}
where $S^{+}_k: = \max(S_k, 0)$ and $S^{-}_k: = - \min(S_k, 0)$.
\end{lemma}

\begin{proof}[Proof of Proposition \ref{prop:main}]
Part $(i)$ is a consequence of the monotone decreasing nature of the convergence \eqref{eq:noninc}: the almost sure convergence of $\widetilde{W}_{k,n}$ to $W_k$  is dominated by $\widetilde{W}_{k,k}$, 
which is easily seen to have a finite $p^{th}$ moment if so does $|X|$. The argument for $D_{k,n}$ is similar.

\quad For part $(ii)$, by Lemma \ref{lem:orderid} and Lemma \ref{lem:expmm}, we have for any distribution of increments $X$ with $\mathbb{E}|X| < \infty$,
\begin{equation}
\label{eq:pwd1}
\mathbb{E}  M_{k,n} = \sum_{j=1}^k \frac{ \mathbb{E} (S_j^+) }{j} - \sum_{j=1}^{n-k} \frac{ \mathbb{E} (S_j^-) }{j},
\end{equation}
and hence by differencing
\begin{equation}
\label{eq:pwd3}
\mathbb{E} D_{k,n}  = \frac{ \mathbb{E} S_k^+ }{k} + \frac{ \mathbb{E} S_{n-k+1} ^{-}}{n-k+1}.
\end{equation}
The formula \eqref{eq:pwd3} implies that $\mathbb{E} D_{n-k+1,n} = \mathbb{E} D_{k,n}$, which is consistent with the
reversibility of spacings \eqref{eq:revspacings}.
The conclusion of part $(ii)$ follows from \eqref{eq:pwd3} and the well known fact \cite[Section 4.7]{Durrettbook} that $S_n/n$ is a reversed martingale 
so $S_{n-k+1} ^{-}/(n-k+1)$ converges as $n \to \infty$ both almost surely and in $L^1$ to $\mathbb{E}(X|\mathcal{E})^-$.
Moreover, if $\mathbb{E}(X|\mathcal{E}) = 0$ almost surely and $\mathbb{E}X^2 < \infty$, 
then $\mathbb{E}D_k = \mathbb{E}S^{+}_k/k$.
Note that $\mathbb{E}(S^+_k/\sqrt{k}) = \mathbb{E}[\mathbb{E}(S_k^{+}/\sqrt{k}|\mathcal{E})]$.
The asymptotics \eqref{eq:dasym0} follows as in the i.i.d. case by de Finetti's theorem and the dominated convergence theorem since
$\mathbb{E}(S^+_k/\sqrt{k}|\mathcal{E}) \le \sqrt{\mathbb{E}(S^2_k/k | \mathcal{E})} = \sqrt{\mathbb{E}(X^2|\mathcal{E})}$.
\end{proof}

\subsection{Proof of Corollary \ref{cor:intro}}
\label{sc3+}
Recall from Section \ref{sc2} that $\Gup_1$ (resp. $\Gdo_1$) is the future minimum time of $\Sup$ (resp. the future maximum time of $\Sdo$).
For $w > 0$, we have:
\begin{align}
\label{eq:pathW}
\mathbb{P}(W_1 > w) & = \mathbb{P}\left(\Sup_{\Gup_1} > w, \, -\Sdo_{\Gdo_1} > w\right) \notag \\
& = \mathbb{P}\left(\Sup_{\Gup_1} > w \right) \mathbb{P}\left(-\Sdo_{\Gdo_1} > w \right),
\end{align}
where the first equality is by the construction in Theorem \ref{thm:main}, and the second equality is due to the independence of the Feller chains $(\Sup, \Sdo)$ of a random walk.
Further by Proposition \ref{prop:Tanakaa} (Tanaka's construction of Feller's chains of a random walk),
\begin{equation}
\label{eq:laddertana}
\Sup_{\Gup_1} \stackrel{d}{=} S_{\tau^+} \quad \mbox{and} \quad \Sdo_{\Gdo_1} = S_{\tau^{-0}}.
\end{equation}
Combining \eqref{eq:pathW} and \eqref{eq:laddertana} yields the equation \eqref{eq:Wdecomp}.
\subsection{Proof of Proposition \ref{prop:mainrate}}
\label{sc33}
Recall the definition of $M_{>n}$ from \eqref{eq:M>n}, and it follows from the coupling bound \eqref{eq:coupling} that
\begin{align*}
\mbox{TV}(K, n) & \le \mathbb{P}(M_{>n} \le W_K) \\
                           & \le \mathbb{P}\left(\min_{k \ge N^{+}_n} \Sup_k \le W_K \right) + \mathbb{P}\left(\min_{k \ge N^{-}_n} (-\Sdo_k) \le W_K \right).
\end{align*}
For $k \ge 1$, let $W^{\downarrow}_k: = M_k(\{-\Sdo_n, \, n \ge 0\})$ be the $k^{th}$ smallest element in the chain $- \Sdo$. 
By definition, for each finite $K$ we have $W_K \le W^{\downarrow}_K$ almost surely. 
Therefore,
\begin{equation}
\label{eq:futurebound}
\mathbb{P}\left(\min_{k \ge N^{+}_n} \Sup_k \le W_K \right) \le  \mathbb{P}\left(\min_{k \ge N^{+}_n} \Sup_k \le W^{\downarrow}_K \right).
\end{equation}
By Theorem \ref{thm:Bertoin93}, the two Feller chains $\Sup$ and $\Sdo$ are independent, and so are $\min_{k \ge N^{+}_n} \Sup_k$ and $W^{\downarrow}_K$.
Thus, the key to the analysis of \eqref{eq:futurebound} is to bound $\mathbb{P}(\min_{k \ge N^{+}_n} \Sup_k \le x)$ for any fixed $x > 0$.
Introduce an auxiliary quantity $A(n)$ such that
\begin{equation*}
A(n) \rightarrow \infty \quad \mbox{and} \quad A(n)/\sqrt{n} \rightarrow 0 \quad \mbox{as } n \rightarrow \infty,
\end{equation*}
which we will determine later.
Note that
\begin{equation}
\label{eq:boundx}
\mathbb{P}\left(\min_{k \ge N^{+}_n} \Sup_k \le x\right) \le \mathbb{P}\left(\Sup_{N^{+}_n} \le  A(n)\right)
+ \mathbb{P}\left(\Sup_{N^{+}_n} >  A(n), \, \min_{k \ge N^{+}_n} \Sup_k \le  x\right).
\end{equation}
Now we treat the two terms on the r.h.s. of \eqref{eq:boundx} separately. 
First by Lemma \ref{lem:Fellern}, $\Sup_{N^{+}_n} \stackrel{d}{=} \max_{0 \le k \le n}S_k$.
By the KMT embedding \cite{KMT75}, there is a coupling between the random walk $(S_k, \, k \ge 0)$ and Brownian motion $(B_t, \, t \ge 0)$ such that 
\begin{equation*}
\mathbb{P}\left(\max_{0 \le k \le n} |S_k - B_k| > C_1 \log n \right) \le n^{-\frac{1}{2}} \quad \mbox{for some } C_1 > 0.
\end{equation*}
As a consequence,
 \begin{align}
 \label{eq:term1}
\mathbb{P}\left(\Sup_{N^{+}_n} \le  A(n)\right) \notag &= \mathbb{P}\left(\max_{0 \le k \le n} S_k \le A(n) \right)  \notag\\
 & \le n^{-\frac{1}{2}} + \mathbb{P}\left(\max_{0 \le k \le n} B_k \le C_1 \log n + A(n) \right) \notag \\
 & \le n^{-\frac{1}{2}} +  C_2  n^{-\frac{1}{2}}(\log n + A(n)) \le C_3  n^{-\frac{1}{2}}(\log n + A(n)),
 \end{align}
for some $C_2, C_3> 0$.
Now let $\tau^{\uparrow}_{A(n)}: = \inf\{k: \Sup_k > A(n)\}$ the first time at which the upward Feller chain $\Sup$ enters $(A(n), \infty)$.
It is easy to see that
$\left\{\Sup_{N^{+}_n} >  A(n), \, \min_{k \ge N^{+}_n} \Sup_k \le  x \right\} \subset \left\{\tau^{\uparrow}_{A(n)} < \infty, \, \min_{k \ge \tau^{\uparrow}_{A(n)} } \Sup_k \le  x \right\}$.
By conditioning on the value of $\Sup_{\tau^{\uparrow}}$, we get
\begin{align*}
\mathbb{P}\left(\Sup_{N^{+}_n} >  A(n), \, \min_{k \ge N^{+}_n} \Sup_k \le  x\right) & \le \mathbb{P} \left(\tau^{\uparrow}_{A(n)} < \infty, \, \min_{k \ge\tau^{\uparrow}_{A(n)} } \Sup_k \le  x \right) \\
& = \int_{A(n) \vee x}^{\infty} \mathbb{P} \left(\min_{k \ge\tau^{\uparrow}_{A(n)}} \Sup_k \le  x \Bigg| \Sup_{\tau^{\uparrow}_{A(n)}} = y \right)  \mathbb{P}\left(\tau^{\uparrow}_{A(n)} < \infty, \, \Sup_{\tau^{\uparrow}_{A(n)}}  \in dy \right).
\end{align*}
According to \cite[Lemma 4.1 \& A.1]{FC08},
\begin{equation*}
\mathbb{P} \left(\min_{k \ge\tau^{\uparrow}_{A(n)}} \Sup_k \le  x \Bigg| \Sup_{\tau^{\uparrow}_{A(n)}} = y \right)
= \frac{h^{\uparrow}(y) - h^{\uparrow}(y-x)}{h^{\uparrow}(y)},
\end{equation*}
where $h^{\uparrow}(\cdot)$ is defined in Theorem \ref{thm:Bertoin93}.
It is also known (see e.g. \cite[p.2155]{BD94}) that $h^{\uparrow}(\cdot)$ is the renewal function of the ladder height $S_{\tau^+}$.
By subadditivity of the renewal function, we have $h^{\uparrow}(y) - h^{\uparrow}(y-x) \le h^{\uparrow}(x)$.
Since $\mathbb{E}|X|^2 < \infty$, it follows from the key renewal theorem that $h^{\uparrow}(y) \sim C_4 y$ as $y \rightarrow \infty$ for some $C_4 > 0$.
The above observations imply that 
\begin{equation}
\label{eq:term2}
\mathbb{P}\left(\Sup_{N^{+}_n} >  A(n), \, \min_{k \ge N^{+}_n} \Sup_k \le  x\right) \le C_5 \,  h^{\uparrow}(x)A(n)^{-1},
\end{equation}
for some $C_5 > 0$.
Combining \eqref{eq:boundx}, \eqref{eq:term1} and \eqref{eq:term2} yields for $x > 0$,
\begin{align}
\label{eq:impest}
\mathbb{P}\left(\min_{k \ge N^{+}_n} \Sup_k \le x\right) \le
C_3  n^{-\frac{1}{2}}(\log n + A(n)) + C_5 \,  h^{\uparrow}(x)A(n)^{-1}.
\end{align}
By taking $A(n) = n^{\frac{1}{4}}$ (which optimizes the right side of \eqref{eq:impest} relative to $A(n)$), we get
\begin{equation}
\label{eq:impest2}
\mathbb{P}\left(\min_{k \ge N^{+}_n} \Sup_k \le x\right) \le C_6 \left(1 + h^{\uparrow}(x) \right) n^{-\frac{1}{4}},
\end{equation}
for some $C_6 > 0$.
Now by \eqref{eq:futurebound} and \eqref{eq:impest2}, we have
\begin{equation*}
\mathbb{P}\left(\min_{k \ge N^{+}_n} \Sup_k \le W_K \right) \le C_6 \, \left(1 + \mathbb{E}h^{\uparrow}(W^{\downarrow}_K)\right) n^{-\frac{1}{4}}.\end{equation*}
Finally, by Proposition \ref{prop:main} $(i)$ and the fact that $h^{\uparrow}(y)/y^2 \rightarrow 0$ as $y \rightarrow \infty$, we get $\mathbb{E}h^{\uparrow}(W^{\downarrow}_K) < \infty$.
A similar bound is derived for $\mathbb{P}\left(\min_{k \ge N^{-}_n} (-\Sdo_k) \le W_K \right)$, which yields the desired result.

\section{Order statistics of simple symmetric random walks}
\label{sc4}

\quad This section describes the limit process $(W_k, \, k \ge 0)$ of shifted order statistics 
$(M_{k,n} - M_{0,n}, \, 0 \le k \le n)$ derived from the path of a random walk with symmetric $\pm 1$ increments; 
that is $\mathbb{P}(X_i = 1) = \mathbb{P}(X_i = -1) = 1/2$.
According to Theorem \ref{thm:main}, $(W_k, k \ge 0)$ is the increasing rearrangement of the two Feller chains 
$(\Sup_{k}, k \ge 1)$ and $(- \Sdo_{k}, k \ge 0)$.
The key to the analysis is the occupation counting process $(L_{\ell}, \, \ell = 0,1,\ldots)$ defined by
\begin{equation}
\label{eq:Lcount}
L_{\ell} : = \sum_{k = 0}^{\infty} 1(W_k = \ell) = \Lup_{\ell} + \Ldo_{\ell},
\end{equation}
where
\begin{equation}
\label{eq:L+-count}
\Lup_{\ell}: = \sum_{k = 1}^{\infty} 1(\Sup_k = \ell ) \quad \mbox{and} \quad 
\Ldo_{\ell}: = \sum_{k = 0}^{\infty} 1(-\Sdo_k = \ell ).
\end{equation}
Note the convention that the minimal order statistic $W_0 = 0$ is included as a contribution to the count
$\Ldo_0 \ge 1$ of values attaining the minimum, while the $0$ value is deliberately excluded as a 
contribution to $\Lup_0 = 0$.

\quad For the simple symmetric random walk the general formula \eqref{eq:limexps} for $\mathbb{E} D_k$ simplifies to 
\begin{equation}
\label{demoivre}
\mathbb{E} D_k = \frac{ \mathbb{E} S_k^+ }{k } = \frac{1}{2} u_{\floor*{k/2}},
\end{equation}
where $u_m := \mathbb{P} ( S_{2m} = 0 )  = { 2 m \choose m } 2^{-2m} \sim \frac{ 1 } {\sqrt{\pi m } }$ as $m \to \infty$,
is the central term in the Binomial$(2m, 1/2)$ distribution.
The second equality in \eqref{demoivre} follows from Kemperman's formula and the reflection principle.
By Kemperman's formula and time reversal,
\begin{align*}
j \, \mathbb{P}(S_k = j) & = k \, \mathbb{P}(S_1 < j, \ldots, S_{k-1}< j, S_k = j) \\
& = k \, \mathbb{P}(S_1 > 0, \ldots, S_{k-1} > 0, S_k = j).
\end{align*}
By summing both sides over $j > 0$, we get $\mathbb{E}S^+_k/k = \mathbb{P}(S_1 > 0, \ldots, S_k > 0)$ 
which equals to $\frac{1}{2} u_{\floor*{k/2}}$ by the reflection principle (see e.g. \cite[p.77, (3.2)]{Feller1}).
For $\theta > 0$, recall the well known identity:
\begin{equation*}
(1-z)^{-\theta} = 1 + \sum_{k \ge 1} \frac{(\theta)_{k \uparrow}}{k!} z^k,
\end{equation*}
where $(\theta)_{k \uparrow}:= \prod_{i = 0}^{k-1} (\theta + i)$ is the Pochhammer function.
Specializing to $\theta = \frac{1}{2}$ yields the generating function $\sum_{k= 0}^\infty u_k z^k = (1 - z )^{-1/2}$, which then gives easily the generating function
\begin{equation}
\sum_{k= 1}^\infty \mathbb{E} D_k  z^k = \frac{1}{2} \left( ( 1 + z ) ( 1 - z^2)^{-1/2} - 1 \right) .
\end{equation}
Since the increments $X_k$ of $(S_k, \, k \ge 0)$ are restricted to $\pm 1$, the increments $D_k$ of $(W_k, \, k \ge 0)$ are restricted to be $0$ or $1$.
Thus $D_k$ is an indicator random variable, more precisely
\begin{equation}
\label{likeren}
D_k = \sum_{j=0}^\infty 1(  \eta_j = k ) =  1(\eta_j = k \mbox{ for some } j \ge 0 ) \quad \mbox{for } k = 1,2, \ldots,
\end{equation}
where 
\begin{equation}
\label{sigmaj}
\eta_j :=  \sum_{\ell=0} ^j  L_{\ell} = \sum_{k = 1}^\infty 1( W_k \le j ),
\end{equation}
is the number of all order statistics at or below level $j$.
Formula \eqref{likeren} looks as if a process of renewal indicators is constructed from the times $\eta_j$ of renewals generated by spacings $L_{\ell}$ in a delayed renewal process \cite[Chapter VI]{Asmussen03}. 
So the formula \eqref{demoivre} for the indicators $D_k$ defined by \eqref{likeren} parallels the well known role of 
$u_k = \mathbb{P}(S_{2k} = 0 )$ in the renewal indicator process $(1(S_{2k} = 0), \, k \ge 0 )$ whose spacings are independent copies of the time of first return to $0$ for the original random walk (see e.g. \cite[Chapter XIII]{Feller1}).  
However, the stochastic structure of the $(L_{\ell}, \, \ell  \ge 0)$ is more complex than the independent random spacings in a renewal process. 
See also \cite{Dwass67, Nara67, Ocka77, Renault08, RV60} for the distribution of order statistics and related  functionals of simple random walks.

\quad 
Call a process $(Z_k, \, k \ge 0)$ a {\em branching process with immigration distribution $F_0$ and offspring distribution $F_1$} if $(Z_k, \, k \ge 0)$ is a Markov chain with
\begin{equation*}
(Z_{k+1} | Z_0, \ldots , Z_k) \stackrel{d}{=} Y_0 + \sum _{i = 1}^ {Z_k} Y_i \quad k \ge 0,
\end{equation*}
where $X_0, X_1, X_2, \ldots$ are independent, with $Y_0$ distributed according to $F_0$ and $Y_i$ for $ i \ge 1$ distributed according to $F_1$.
In the present setting, $Y_0$ will always be a constant, either $0$ or $1$ or $2$, in which case $(Z_k, \, k \ge 0)$ is said to have either {\em no immigration} or {\em single} or {\em double immigration} as the case may be.
We say $Y$ has the Geo$_0(p)$ distribution on $ \{0,1, \ldots\}$ if $\mathbb{P}(Y = k) = p(1-p)^k$ for $k \ge 0$, 
and $Y$ has the Geo$_1(p)$ distribution on $ \{1,2, \ldots\}$ if $\mathbb{P}(Y = k) = p(1-p)^{k-1}$ for $k \ge 1$.
The following theorem summarizes some results underlying the construction of the occupation counting process $(L_{\ell}, \, \ell \ge 0)$.

\begin{theorem}
\label{thm:orderSSW}
Let $(L_{\ell}, \, \ell \ge 0)$ be the occupation counting process defined by \eqref{eq:Lcount}. Then
\begin{itemize}[itemsep = 3 pt]
\item[(i)]
$L$ may be represented as
\begin{equation}
L_{\ell} = Z_{\ell - 1} + Z_{\ell} - 2, \quad \ell \ge 0,
\end{equation}
where  $(Z_{\ell}, \, \ell = -1,0,1, \ldots)$ is a branching process with double immigration and Geo$_0(1/2)$ offspring distribution, with the convention $Z_{-1} = 1$.
In particular, $L_{0} = Z_{0} - 1$ has Geo$_1(1/2)$ distribution:
\begin{equation}
\label{geomhits}
\mathbb{P}( L_0 = k ) = \mathbb{P}(Z_0 = k + 1) = 2^{-k} \quad \mbox{for } k \ge 1.
\end{equation}
\item[(ii)]
For each $\ell \ge 1$, the conditional distribution of $L_{\ell}$ given $L_0 = k$ is that of 
\begin{equation}
G_0(\ell) + G_k (\ell) + \sum_{i=1}^{k - 1} B_i(\ell) G_i(\ell)
\end{equation}
where the $(G_i(\ell), \, 0 \le i \le k)$ are independent Geo$_1(\frac{1}{2 \ell})$ variables, and
the $(B_i(\ell), \, 1 \le i \le k-1)$ are $k-1$ independent Bernoulli$(\frac{1}{\ell})$ variables, independent also of the $(G_i(\ell), \, 0 \le i \le k)$.
\item[(iii)]
For each $k \ge 1$, $\eta_k$ defined by \eqref{sigmaj} has the generating function
\begin{equation}
\label{siggf}
\mathbb{E} z^{\eta_k} = \frac{ 1 } { V_k(1/z) V_{k+1}(1/z) },
\end{equation}
where $V_k$ is the sequence of Chebyshev polynomials of the third kind, defined by the Chebyshev recurrence
\begin{equation}
\label{chebrec}
V_{k}(x) := 2 x V_{k-1}(x) - V_{k-2}(x)  \quad k \ge 2,
\end{equation}
with initial condition $V_0(x):= 1$ and $V_1(x):= 2 x - 1$.
\end{itemize}
\end{theorem}

\quad To prove Theorem \ref{thm:orderSSW}, we need the following two lemmas regarding the Feller chain $\Sup$ derived from a simple symmetric random walk.

\begin{lemma} \cite{Knight63}
\label{lem:Knight}
Let $\Sup$ and $\Sdo$ be the upward and downward Feller chains derived from a simple symmetric random walk $S$. 
For $\ell \ge 0$, let
\begin{equation}
\Zup_\ell := \sum_{k = 0}^\infty 1(\Sup_{k} = \ell, \, \Sup_{k+1} = \ell + 1 ), \, 
\Zdo_\ell := \sum_{k = 0}^\infty 1(- \Sdo_{k} = \ell, \, -\Sdo_{k+1} = \ell + 1 ),
\end{equation}
be the total number of upcrossings of $\Sup$ and $-\Sdo$ from $\ell$ to $\ell + 1$.
Then we have:
\begin{itemize}[itemsep = 3 pt]
\item[(i)]
The sequence $(\Zup_\ell, \, \ell \ge 0)$ is a branching process with single immigration and Geo$_0(1/2)$ offspring distribution.
\item[(ii)]
The sequence $(\Zdo_\ell, \, \ell \ge 0)$ is a branching process with single immigration and Geo$_0(1/2)$ offspring distribution, which evolves independently of $\Zup$, but with a time shift forward by one step. 
So $\Zdo_0$ starts with the Geo$_1(1/2)$ distribution.
\end{itemize}
\end{lemma}

\begin{lemma}
\label{lem:3iden}
For $k \ge 1$ let $\tau_k:= \inf \{n \ge 1 : S_n = k \}$ be the first passage time to level $k > 0$ for the simple symmetric random walk $S$. 
Let $N_{0+}(n):= \sum_{j=1}^n 1 ( S_j \ge 0 )$
be the number of non-negative steps of the walk up to time $n$, let
$\overline{\tau}_k:= \inf \left\{n \ge 1 : S_n - \underline{S}_n = k \right\}$
be the first passage time to level $k$ for the reflected random walk $(S_n - \underline{S}_n, \, n \ge 0)$, 
and let $\eta^{\uparrow}_k := \sum_{j = 1 }^\infty 1 ( \Sup_j \le k )$
be the total occupation time of levels $1$ through $k$ for the upward Feller chain.
Then we have:
\begin{itemize}[itemsep = 3 pt]
\item[(i)]
The three random variables $N_{0+}(\tau_k)$, $\overline{\tau}_k$ and $\eta^{\uparrow}_k$ have identical distributions.
\item[(ii)]
The common distribution of these random variables has the generating function
\begin{equation}
\label{chebgf}
g_k(z):= \mathbb{E} z^ { N_{0+}(\tau_k) } = \mathbb{E} z^ { \overline{\tau}_k } = \mathbb{E} z^ {  \eta^{\uparrow}_k} = \frac{1}{V_k(1/z)},
\end{equation}
where $V_k$ is the sequence of Chebyshev polynomials of the third kind defined by \eqref{chebrec}.
\end{itemize}
\end{lemma}

\begin{proof}
$(i)$ It follows from \cite{Pitman75} that there is the identity in distribution:
\begin{equation*}
(\Sup_j, \, 0 \le j \le \sigma_k) \stackrel{d}{=} (k - S_{\tau_k - j}, \, 0 \le j \le \tau_k),
\end{equation*}
where $\sigma_k:=\max\{j, \, \Sup_j =k\}$ is the last time at which $\Sup$ visits the state $k$.
As a result,
$N_{0+}(\tau_k) = \sum_{j = 0}^{\tau_k-1} 1(k - S_{\tau_k - j} \le k) \stackrel{d}{=} \sum_{j = 0}^{\sigma_k-1} 1(\Sup_j \le k) = \eta^{\uparrow}_k$, where the last equality follows from the fact that $\Sup$ does not visit the state $k$ after time $\sigma_k$.
The further equality in distribution of $N_{0+}(\tau_k)$ and $\overline{\tau}_k$ follows easily from the stronger result that if
$\nu_0:= 0$ and $\nu_{k}$ for $k > 0$  is the first time $n$ such that $N_{0+}(n) = k$, then $(S_{\nu_k}, \, k \ge 0 )$, which is just the subsequence of non-negative terms of $S$, has the same distribution as $(S_n -\underline{S}_n, \, n \ge 0)$. 
So each process is just simple symmetric random on the non-negative integers, in which the transition probabilities out of state $0$ are $P(0,0) = P(0,1) = 1/2$.

\quad $(ii)$ Let $g_k(z):= \mathbb{E} z^ { N_{0+}(\tau_k) }$.
Let $(\mathcal{F}_k, \, k \ge 0)$ be the filtration of the random walk $S$.
Note that $S_{\tau_{k-1} + 1} = k$ or $k-2$ with equal probabilities $\frac{1}{2}$.
By the strong Markov property of $S$ at time $\tau_{k-1}$, we have:
\begin{equation*}
(N_{0+}(\tau_k) | \mathcal{F}_{\tau_{k-1}}) = \left\{ \begin{array}{lcl}
N_{0+}(\tau_{k-1}) + 1 & \mbox{with probability } \frac{1}{2}, \\ 
N_{0+}(\tau_{k-1}) + 1 + N_{0+}(\tau_{k-2:k}) & \mbox{with probability } \frac{1}{2},
\end{array}\right.
\end{equation*}
where $N_{0+}(\tau_{k-2:k})$ is the number of non-negative steps of the walk starting at $k-2$ until it first visits $k$, independent of $N_{0+}(\tau_{k-1})$.
Therefore, $g_k(z) = \frac{1}{2} z g_{k-1}(z) + \frac{1}{2} z g_{k-1}(z) \mathbb{E}z^{N_{0+}(\tau_{k-2:k})}$.
Further by the strong Markov property of $S$ at time $\tau_{k-2}$, we have $g_{k-2}(z) \mathbb{E}z^{N_{0+}(\tau_{k-2:k})} = g_{k}(z)$.
By eliminating the term $\mathbb{E}z^{N_{0+}(\tau_{k-2:k})}$, we get the recursion
\begin{equation}
\label{gfrec}
g_k(z) = g_{k-1}(z) z \left( \frac{1}{2} + \frac{1}{2} \frac{ g_k(z) }{ g_{k-2}(z) } \right) \quad \mbox{for } k \ge 2,
\end{equation}
with the convention that $g_0(z):= 1$ and $g_{-1}(z):=1$ (so $g_1(z) = \frac{z}{2-z}$). 
Let $V_k(x):= 1/g_k(1/x)$. 
Then \eqref{gfrec} reduces easily to \eqref{chebrec}, with $V_0(x) = 1$ and $V_1(x) = 2x-1$.
\end{proof}

\begin{proof}[Proof of Theorem \ref{thm:orderSSW}]
Part $(i)$ follows from Lemma \ref{lem:Knight}, and the path construction
\begin{equation}
\Lup_{\ell} = \Zup_{\ell-1} + \Zup_{\ell} - 1 \quad  \mbox{and} \quad  \Ldo_{\ell} = \Zdo_{\ell-1} + \Zdo_{\ell} - 1,  \quad \ell \ge 0,
\end{equation}
with the conventions $\Zup_{-1} = 0$ and $\Zdo_{-1} = 1$ for $\ell = 0$.
Part $(ii)$ can be read from R\'{e}v\'{e}sz \cite[Lemma 12.5]{Revesz13}.
For $S$ a simple symmetric walk, 
let $\tau_0:= \inf\{n \ge 1: S_n=0\}$ be the first time at which $S$ returns to $0$, and $\xi_\ell: = \sum_{k = 1}^{\tau_0}1(S_k = \ell)$ be the number of visits to the state $\ell \ge 0$ up to time $\tau_0$.
It was shown that
\begin{equation}
\label{eq:idRevesz}
\left(\xi_{n - \ell} \bigg| \max_{0 \le k \le \tau_0} S_k = n, \, \xi_n = k\right) \stackrel{d}{=} G_0(\ell,n) + G_k (\ell,n) + \sum_{i=1}^{k - 1} B_i(\ell,n) G_i(\ell,n),
\end{equation}
where the $(G_i(\ell,n), \, 0 \le i \le k)$ are independent Geo$_1(\frac{n}{2 \ell (n - \ell)})$ variables, and
the $(B_i(\ell,n), \, 1 \le i \le k-1)$ are $k-1$ independent Bernoulli$(1 - (1-\frac{1}{\ell})\frac{n}{n-1})$ variables, independent also of the $(G_i(\ell,n), \, 0 \le i \le k)$.
By taking $n \to \infty$, the left side of \eqref{eq:idRevesz} converges in distribution to $(L_{\ell} | L_0 = k)$, and $G_i(\ell,n)$ and $B_i(\ell,n)$ converge in distribution to $G_i(\ell)$ and $B_i(\ell)$ respectively, hence the conclusion.
For part $(iii)$, let $\eta^{\uparrow}_k: = \sum_{j = 1}^{\infty}1(\Sup_j \le k)$. It is easy to see that
$\eta_k \stackrel{d}{=} \eta^{\uparrow}_k + \widetilde{\eta}^{\uparrow}_{k+1}$,
where $\widetilde{\eta}^{\uparrow}_{k+1}$ is a copy of $\eta^{\uparrow}_{k+1}$, independent of $\eta^{\uparrow}_k$.
It suffices to apply Lemma \ref{lem:3iden} to conclude.
\end{proof}

\section{Order statistics of Gaussian walks}
\label{sc5}

\quad This section provides further analysis of limiting order statistics of a random walk with i.i.d. Gaussian increments;
that is $\mathbb{P}(X_i \in dx) = (2 \pi \sigma^2)^{-\frac{1}{2}} \exp \left(-\frac{x^2}{2 \sigma^2}\right) dx$ for some $\sigma^2 > 0$.
As discussed in the introduction, a random walk $(S_k, \, 0 \le k \le n)$ with Gaussian increments is embedded in a Brownian motion $(B(t), \,  t \ge 0)$ via $S_k = B(\sigma k)$.
There is one more variable to be entered into the mix; that is
\begin{equation}
M_{-,n}:= \min_{0 \le t \le \sigma n} B(t).
\end{equation}

So by definition, $M_{-,n} \le M_{0,n} \le \cdots \le M_{n,n}$ are the ranked values of Brownian motion on $[0, \sigma]$ evaluated at a grid of $n+2$ fixed times $0, \sigma, 2 \sigma, \ldots, \sigma n$ at which the random walk $S_k:= B(\sigma k)$ is embedded, 
and one extra random time $\alpha(\sigma n): = \inf\{t \le \sigma n: B(t) = \underline{B}(\sigma n)\}$, the almost surely unique random time at which $B$ attains its minimum on $[0,\sigma n]$.
It is a key observation that the structure of limit distributions for differences of order statistics $M_{k,n} - M_{0,n}$ for $k = O(\sqrt{n})$ as $n \to \infty$
is easily understood by first considering the differences $M_{k,n} - M_{-,n}$,  
from which the $M_{k,n} - M_{0,n}$ can be recovered using Theorem \ref{thm:Gaussianmain} as
\begin{equation}
M_{k,n} - M_{0,n} = \left(M_{k,n} - M_{-,n} \right)- \left(M_{0,n} - M_{-,n}\right).
\end{equation}

\quad The distribution of $M_{0,n} - M_{-,n}$ was studied in Asmussen, Glynn and Pitman \cite{AGP95} in the Gaussian case,  
where $M_{0,n} - M_{-,n}$ was interpreted as the discretization error in approximating a reflecting random walk by a natural Euler approximation scheme. 
The  key idea was that when the Gaussian random walk $S$ is embedded in Brownian motion $B$ as $S_k = B(\sigma k)$ for $0 \le k \le n$, 
and $n$ is large, the discretization error $M_{0,n} - M_{-,n}$ is with overwhelming probability determined by the behavior of $B$ in a time window of width $o(n)$ around the minimum time $\alpha(\sigma n)$.
According to basic decompositions of a Brownian path at a local minimum, due to Williams \cite{Williams} and Denisov \cite{Denisov},
after centering at $(\alpha, B(\alpha))$ for $\alpha = \alpha(\sigma n)$ the height of the Brownian path above its minimum value
becomes locally well approximated by two independent copies of a BES(3) process, say $(R_3(t), t \ge 0)$, and $(\Rcap_3(t), \, t \ge 0)$, joined back to back, call it the Brownian valley:
\begin{equation}
\Rcup(t):= R_3(t) 1 ( t \ge 0 ) + \Rcap_3(t) 1 ( t < 0), \quad \mbox{for } - \infty < t < \infty.
\end{equation}


\quad It is easily shown that for each fixed $\varepsilon >0$, there exists $a(\varepsilon) > 0$ and $n(\varepsilon)$ so that for $n \ge n(\varepsilon)$
the probability that the Gaussian walk attains its minimum over $\{0,1, \ldots, n\}$ at a time $k$ with $|k - \alpha(n)| < a(\varepsilon)$ is at least $1-\varepsilon$.
As shown in \cite[Theorem 1]{AGP95}, it follows the discretization error for the standard Gaussian random walk $S$ is described by
\begin{equation}
\label{dkzero}
M_{0,n} -M_{-,n} = \min_{0 \le k \le n }S_k - \min_{0 \le t \le n } B_t \stackrel{d}{\longrightarrow}  \min _{k \in \mathbb{Z}} \Rcup(U + k)
\end{equation}
with $U$ Uniform$[0,1]$ and independent of $\Rcup$. 
The same argument can establish more, which is stated in the following theorem.

\begin{theorem}
\label{thm:Gaussianmain}
Let $(S_k, \, k \ge 0)$ be a Gaussian random walk, for increments with variance $\sigma^2$.
For $0 \le k \le n$, let $M_{k,n}$ be the sequence of order statistics defined by \eqref{eq:Mkn} of the $n$-step walk $(S_k, \, 0 \le k \le n)$.
Then for every $m = 1,2, \ldots$, there is the convergence of joint distributions
 \begin{equation}
(M_{k,n} -M_{-,n}, \, 0 \le k \le m)  \stackrel{d}{\longrightarrow}   (\sigma M_{k,\infty}, \, 0 \le k \le m)
\end{equation}
where $M_{k, \infty}: = M_{k}(\{ \Rcup(U + j), \, j \in \mathbb{Z} \})$ is the sequence of order statistics of 
the values $\{ \Rcup(U + j), \, j \in \mathbb{Z} \}$, and $U$ with uniform distribution on $[0,1]$ is independent of $\Rcup$.
Moreover, the distribution of $M_{0, \infty}: = \min_{k \in \mathbb{Z}} \Rcup(U+k)$ is given by
\begin{equation}
\label{eq:tail}
\mathbb{P}\left(M_{0, \infty} >a\right)= \int_0^1 G^a(u) G^a(1-u) du,
\end{equation}
with
\begin{equation}
\label{eq:Ga}
G^a(u) = K^a(u) \prod_{k = 0}^{\infty} \frac{H^a(k+u)}{K^a(k+u)}, 
\end{equation}
where
\begin{equation}
\label{eq:Ka}
K^a(t) = \sqrt{\frac{2}{\pi t}} a e^{-\frac{a^2}{2t}} + \erfc\left(\frac{a}{\sqrt{2t}} \right),
\end{equation}
and 
\begin{align}
\label{eq:Haformula}
H^a(t) & = \frac{a}{\sqrt{2 \pi (1+t)}}e^{-\frac{a^2}{2(1+t)}}\left( \erfc\left(\frac{a}{\sqrt{2t(t+1)}} \right)  + \erfc\left(\frac{a(2t+1)}{\sqrt{2t(t+1)}} \right) \right) \notag \\
& \quad \, + \frac{1}{\pi \sqrt{t}} \left(e^{-\frac{a^2}{2t}} - e^{-\frac{a^2(4t+1)}{2t}} \right) 
+ \frac{a}{\sqrt{2 \pi t}} e^{-\frac{a^2}{2t}}(1 + \erfc(\sqrt{2}a)) + \erfc \left( \frac{a}{\sqrt{2t}} \right)  \\
& \quad \,  +  \erfc \left( \frac{a}{\sqrt{2(t+1)}} \right) -T\left(\frac{a}{\sqrt{t+1}}, \frac{1}{\sqrt{t}}\right)- T\left(\frac{a}{\sqrt{t+1}}, \frac{2t+1}{\sqrt{t}}\right) -   T\left(\frac{a}{\sqrt{t}}, 2 \sqrt{t}\right). \notag
\end{align}
Here $\erfc(x) := 1 -  \frac{2}{\sqrt{\pi}} \int_0^{x} e^{-t^2} dt$ is the complementary error function,
and $T(h,a)$ is Owen's T function defined by 
\begin{equation}
\label{eq:OwenT}
T(h,a): = \frac{1}{2 \pi} \int_0^a \frac{e^{\frac{1}{2} h^2(1+x^2)}}{1+x^2} dx.
\end{equation}
\end{theorem}

\quad Note that in the Gaussian case, the limit in distribution conjectured by 
Schehr and Majumdar \cite{SM12} for $D_k$ and $W_k$ can be read from the above theorem, with the process
$(D_{k},\,  k \ge 1)$ identified as $D_k = M_{k,\infty} - M_{k-1,\infty}$, and the partial sums $(W_k, \, k \ge 0)$ identified as $W_k = M_{k,\infty} - M_{0,\infty}$.
It is a key point of the present approach that the enlargement of the probability space to include an auxilliary Brownian motion enables a simple description of the limit process, involving a single additional variable $M_{0,\infty}$ which
places the limiting order statistics in their natural environment, which is a process of random sampling of points in the Brownian valley.
See also \cite{Chen11, Ivanovs18} for related works on L\'evy processes and random walks with heavy tailed increments.

\quad As discussed in \cite{AGP95}, in the Gaussian case, it is a difficult problem to describe the even the limit distribution of the $M_{0,\infty}$ in a simple way, see e.g. formula \eqref{eq:tail} -- \eqref{eq:Haformula}. 
While bounds on the distribution of $M_{0,\infty}$ allow to show that this variable has finite positive moments of all orders, the only exact result found is the mean:
\begin{equation}
\mathbb{E} M_{0,\infty} = - \frac{ \sigma} {\sqrt{2 \pi } } \zeta(1/2) = 0.5826 \cdots \sigma,
\end{equation}
where $\zeta(\cdot)$ is the Riemann zeta function.
In view of \eqref{eq:tail}, we get the following non-trivial identity:
\begin{equation}
\label{eq:zetaiden}
\int_0^{\infty} \int_0^1 K^a(u)K^a(1-u) \prod_{k = 0}^{\infty} \frac{H^a(k+u) H^a(k+1-u)}{K^a(k+u) K^a(k+1-u)} du da = - \frac{1}{\sqrt{2 \pi}} \zeta\left(\frac{1}{2}\right).
\end{equation}

Further computations of explicit distributions in this case seem very difficult, and it is interesting to see if further moments, or the generating function of $M_{0,\infty}$ can be computed by using the product formula \eqref{eq:tail}.
See also \cite{CP97, JvL072, JvL07} for closely related computations of the mean of the ladder height distribution for a Gaussian random walk, which also involve Riemann's zeta function.

\begin{proof}[Proof of Theorem \ref{thm:Gaussianmain}]
The first part of the theorem is a simple consequence of the proof of \cite[Theorem 1]{AGP95}.
We focus on the second part of the distribution of $M_{0, \infty}$.
The derivation relies on some Bessel process computation, which is of independent interest.
Fix $a > 0$, we aim to compute 
\begin{equation*}
\mathbb{P}\left(M_{0, \infty} >a\right)
:= \mathbb{P}(R_3(k+U) > a \mbox{ and } \Rcap_3(k+1-U)>a,  \, \forall k \ge 0),
\end{equation*}
where $R_3$ and $\Rcap_3$ are two independent BES(3) processes. 
By conditioning on $U$, we get
\begin{equation}
\mathbb{P}\left(M_{0, \infty} >a\right) = \int_0^1 \mathbb{P}(R_3(k+u) > a, \, \forall k \ge 0) \cdot \mathbb{P}(R_3(k+1-u) > a, \, \forall k \ge 0) du.
\end{equation}
For $u \in (0,1)$, by letting $G^a(u) = \mathbb{P}(R_3(u+k) > a, \, \forall k \in \mathbb{N})$,
we get \eqref{eq:tail}.

\quad Now we calculate $G^a(u)$. For $t > 0$, let
$K^a(t): =  \mathbb{P}(R_3(t)>a)$ and  $H^a(t): = \mathbb{P}(R_3(t+1)>a, \, R_3(t)>a)$.
By the Markov property of the BES(3) process, we can express $G^a(u)$ by an infinite product:
\begin{equation*}
G^a(u) = \mathbb{P}(R_3(u) > a) \prod_{k = 1}^{\infty} \mathbb{P}(R_3(k+u) > a| R_3(k-1+u) > a),
\end{equation*}
which leads to \eqref{eq:Ga}.
Recall from \cite[Chapter VI.3]{RY} the entrance law $p_t(x): = \mathbb{P}(R_3(t) \in dx)/dx$, and the transition density $q_t(x, y): = \mathbb{P}(R_3(t+t_0) \in dy|R_3(t_0) = x)/dy$ of the BES(3) process:
\begin{equation}
p_t(x) = \sqrt{\frac{2}{\pi t^3}} x^2 \exp \left(-\frac{x^2}{2t} \right), \, x > 0,
\end{equation}
\begin{equation}
q_t(x,y) = \frac{x^{-1}y}{\sqrt{2 \pi t}} \left[\exp\left(-\frac{(y-x)^2}{2t} \right) - \exp\left(-\frac{(y+x)^2}{2t} \right) \right], \, x,y >0.
\end{equation}
It is easy to see that for $t > 0$, $K^a(t)$ is given by \eqref{eq:Ka}.
It remains to compute $H^a(t)$ for $t > 0$.
By definition,
\begin{equation}
\label{eq:Ha}
H^a(t) = \frac{1}{\pi \sqrt{t^3}} \int_a^{\infty}  x \exp\left(-\frac{x^2}{2t} \right) \int_a^{\infty} y \left[\exp\left(-\frac{(y-x)^2}{2} \right) - \exp\left(-\frac{(y+x)^2}{2} \right) \right] dy dx.
\end{equation}
Unfortunately, Mathematica did not provide any simplification for \eqref{eq:Ha}.
Here we simplify \eqref{eq:Ha} by hand using some special functions.
First, 
\begin{multline}
\int_a^{\infty}  y \left[\exp\left(-\frac{(y-x)^2}{2} \right) - \exp\left(-\frac{(y+x)^2}{2} \right) \right] dy \\
= e^{ -\frac{1}{2} (a-x)^2} - e^{-\frac{1}{2} (a+x)^2} + \sqrt{\frac{\pi}{2}} x \erfc\left(\frac{a-x}{\sqrt{2}} \right) +
\sqrt{\frac{\pi}{2}} x \erfc\left(\frac{a+x}{\sqrt{2}} \right).
\end{multline}
So we need to compute the following integrals:
\begin{align*}
& \mbox{I}: = \int_a^{\infty} x \exp\left(- \frac{1}{2}(a-x)^2 - \frac{x^2}{2t} \right) dx, \quad  \mbox{II}: = \int_a^{\infty} x \exp\left(- \frac{1}{2}(a+x)^2 - \frac{x^2}{2t} \right) dx, \\
& \mbox{III}: = \int_a^{\infty} x^2 \exp\left(-\frac{x^2}{2t} \right) \erfc\left(\frac{a-x}{\sqrt{2}}\right) dx, \quad
\mbox{IV}:= \int_a^{\infty} x^2 \exp\left(-\frac{x^2}{2t} \right) \erfc\left(\frac{a+x}{\sqrt{2}}\right) dx.
\end{align*}

By direct computations, we get
\begin{equation}
\label{eq:I}
\mbox{I} = a\sqrt{\frac{\pi t^3}{2(t+1)^3}} \exp\left(-\frac{a^2}{2(t+1)} \right) \erfc\left(\frac{a}{\sqrt{2t(t+1)}} \right) + \frac{t}{t+1} \exp\left( - \frac{a^2}{2t}\right).
\end{equation}
\begin{equation}
\label{eq:II}
\mbox{II} = -a\sqrt{\frac{\pi t^3}{2(t+1)^3}} \exp\left(-\frac{a^2}{2(t+1)} \right) \erfc\left(\frac{a(2t+1)}{\sqrt{2t(t+1)}} \right) + \frac{t}{t+1} \exp\left(-\frac{a^2(4t+1)}{2t} \right).
\end{equation}

Introduce $F(x) : = -\frac{t^{3/2}}{2 \sqrt{t+1}} e^{-\frac{a^2}{2(t+1)}} \erfc \left(- \frac{(1+t)x + at/\sqrt{2}}{\sqrt{t(t+1)}} \right) - \frac{t}{2} e^{-\frac{x^2}{t}} \erfc\left(x + \frac{a}{\sqrt{2}}\right)$.
By integration by parts, we get
$\mbox{III}/\sqrt{8} =  \int_{-\infty}^{-\frac{a}{\sqrt{2}}} x^2 \exp\left(-\frac{x^2}{t} \right) \erfc\left(x+ \frac{a}{\sqrt{2}}\right)dx = [x F(x)]^{-\frac{a}{\sqrt{2}}}_{-\infty} -  \int_{-\infty}^{-\frac{a}{\sqrt{2}}} F(x) dx$. 
Note that 
$[x F(x)]^{-\frac{a}{\sqrt{2}}}_{-\infty} = \frac{at^{3/2}}{\sqrt{8 (t+1)}} e^{-\frac{a^2}{2(t+1)}} \erfc\left( \frac{a}{\sqrt{2t(t+1)}} \right) +  \frac{at}{\sqrt{8}} e^{-\frac{a^2}{2t}}$,
\begin{multline*}
-  \int_{-\infty}^{-\frac{a}{\sqrt{2}}} F(x) dx = \frac{t^{3/2}}{2 \sqrt{t+1}} e^{-\frac{a^2}{2(t+1)}} \int_{-\infty}^{-\frac{a}{\sqrt{2}}} \erfc \left(- \frac{(1+t)x + at/\sqrt{2}}{\sqrt{t(t+1)}} \right) dx
\\ + \frac{t}{2}  \int_{-\infty}^{-\frac{a}{\sqrt{2}}}  e^{-\frac{x^2}{t}} \erfc\left(x+ \frac{a}{\sqrt{2}}\right) dx
\end{multline*}
and
\begin{equation*}
 \int_{-\infty}^{-\frac{a}{\sqrt{2}}} \erfc \left(- \frac{(1+t)x + at/\sqrt{2}}{\sqrt{t(t+1)}} \right) dx
 = \sqrt{\frac{t}{t+1}} \left[- \frac{a}{\sqrt{2t(t+1)}}\erfc\left( \frac{a}{\sqrt{2t(t+1)}}\right) + \frac{1}{\sqrt{\pi}}e^{- \frac{a^2}{2t(t+1)}}\right].
\end{equation*}
It remains to evaluate $\int_{-\infty}^{-\frac{a}{\sqrt{2}}}  e^{-\frac{x^2}{t}} \erfc\left(x+ \frac{a}{\sqrt{2}}\right) dx$.
Let $(X, Y)$ be distributed as independent standard normal (i.e. bivariate normal with mean $0$, variance $1$ and correlation $0$).
We have
\begin{align*}
\int_{-\infty}^{-\frac{a}{\sqrt{2}}}  e^{-\frac{x^2}{t}} \erfc\left(x+ \frac{a}{\sqrt{2}}\right) dx
& = \sqrt{\pi t} \int_{\infty}^{-a} \frac{1}{\sqrt{2 \pi t}}e^{-\frac{x^2}{2t}} \erfc\left(\frac{x+a}{\sqrt{2}} \right) dx \\
& = \sqrt{\pi t} \,  \mathbb{P}(X > a/\sqrt{t}, \, Y + \sqrt{t}X > a).
\end{align*}
It is well known \cite[(2.1)]{Owen56} that for $(X,Z)$ a bivariate normal with mean $0$, variance $1$ and correlation $r$, 
$\mathbb{P}(X > h, \, Z > k) = \frac{1}{2} (\erfc(h/\sqrt{2}) + \erfc(k/\sqrt{2})) - T\left(h,\frac{k - rh}{h \sqrt{1-r^2}} \right) - T\left(k, \frac{h - rk}{k \sqrt{1 - r^2}}\right)$,
where $T(h,a)$ is defined by \eqref{eq:OwenT}.
Specializing to the case $r = \sqrt{t/(t+1)}$, $h = a/\sqrt{t}$ and $k = a/\sqrt{t+1}$, we get
\begin{equation*}
\int_{-\infty}^{-\frac{a}{\sqrt{2}}}  e^{-\frac{x^2}{t}} \erfc\left(x+ \frac{a}{\sqrt{2}}\right) dx 
= \frac{\sqrt{\pi t}}{2} \left(\erfc \left( \frac{a}{\sqrt{2t}} \right) +  \erfc \left( \frac{a}{\sqrt{2(t+1)}} \right) \right) 
- \sqrt{\pi t} \, T\left(\frac{a}{\sqrt{t+1}}, \frac{1}{\sqrt{t}}\right).
\end{equation*}
Combining all the above calculations, we obtain
\begin{multline}
\mbox{III} = a \sqrt{\frac{t^5}{(t+1)^3}} e^{-\frac{a^2}{2(t+1)}} \erfc\left(\frac{a}{\sqrt{2t(t+1)}} \right) 
+ \left( at + \frac{\sqrt{2} t^2}{\sqrt{\pi}(t+1)}\right) e^{-\frac{a^2}{2t}} \\
+ \sqrt{\frac{\pi t^3}{2}} \left(\erfc \left( \frac{a}{\sqrt{2t}} \right) +  \erfc \left( \frac{a}{\sqrt{2(t+1)}} \right) \right) 
- \sqrt{2 \pi t^3} T\left(\frac{a}{\sqrt{t+1}}, \frac{1}{\sqrt{t}}\right).
\end{multline}
Similarly, we get
\begin{multline}
\mbox{IV} = a \sqrt{\frac{t^5}{(t+1)^3}} e^{-\frac{a^2}{2(t+1)}} \erfc\left(\frac{a(2t+1)}{\sqrt{2t(t+1)}} \right) 
+ at \erfc(\sqrt{2}a) e^{-\frac{a^2}{2t}} - \frac{\sqrt{2} t^2}{\sqrt{\pi}(t+1)}e^{-\frac{a^2(4t+1)}{2t}} \\
+ \sqrt{\frac{\pi t^3}{2}} \left(\erfc \left( \frac{a}{\sqrt{2t}} \right) +  \erfc \left( \frac{a}{\sqrt{2(t+1)}} \right) \right) 
- \sqrt{2 \pi t^3} \left( T\left(\frac{a}{\sqrt{t+1}}, \frac{2t+1}{\sqrt{t}}\right) +   T\left(\frac{a}{\sqrt{t}}, 2 \sqrt{t}\right) \right).
\end{multline}

Bringing all the computations together, we obtain
$H^a(t)  = \frac{1}{\pi t^{3/2}} \left(\mbox{I} - \mbox{II} + \sqrt{\frac{\pi}{2}} \mbox{III} + \sqrt{\frac{\pi}{2}} \mbox{IV} \right)$, which yields \eqref{eq:Haformula}.
\end{proof}

\bigskip
{\bf Acknowledgment}: We thank G. Schehr for stimulating discussions at the early stage of this work.
We also thank two anonymous referees for numerous suggestions which improved the final version of this article.
Wenpin Tang gratefully acknowledges financial support through an NSF grant DMS-2113779 and through a start-up grant at Columbia University.

\bibliographystyle{abbrv}
\bibliography{unique}
\end{document}